\documentclass[11pt]{amsart}%
\usepackage{amsmath}
\usepackage{graphicx}
\usepackage{xcolor}
\usepackage{amsfonts}
\usepackage{amssymb}%
\providecommand{\U}[1]{\protect\rule{.1in}{.1in}}
\newtheorem{theorem}{Theorem}[section]
\theoremstyle{plain}

\newtheorem{corollary}[theorem]{Corollary}

\newtheorem{lemma}[theorem]{Lemma}

\newtheorem{proposition}[theorem]{Proposition}
\newtheorem{remark}[theorem]{Remark}

\numberwithin{equation}{section}
\begin{document}
\title{Generalized Jordan derivations of unital algebras}
\author{Dominik Benkovi\v{c}}
\email{dominik.benkovic@um.si}

\author{Mateja Gra\v{s}i\v{c}}
\email{mateja.grasic@um.si}
\address{University of Maribor, Faculty of Natural Sciences and Mathematics, 2000
Maribor, Slovenia}
\address{Institute of Mathematics, Physics and Mechanics, 1000 Ljubljana, Slovenia }


\begin{abstract}

Let $A$ be a unital algebra over a field $F$ with $\operatorname*{char} (F)\neq2$. In this paper we introduce a new concept of a generalized Jordan derivation, covering Jordan centralizers and Jordan derivations, as follows: a linear map $f:A\rightarrow A$ is a generalized Jordan derivation if there exist linear maps $g;h:A\rightarrow A$ such that $f\left(  x\right)  \circ y+x\circ g\left(  y\right)  =h\left(  x\circ y\right)  $ for all $x,y\in A$ (here $x\circ y=xy+yx$). 
Our aim is to give the form of map $f$ in terms of the so called quasi Jordan centralizers and quasi Jordan derivations. In addition, a characterization of such maps is presented.

\end{abstract}
\maketitle

\section{Introduction}

Let $\mathcal{A}$ be a unital (associative) algebra over a field $F$ with $\operatorname*{char}\left(  F\right)  \neq2$. By $x\circ y=xy+yx$ we denote the Jordan product of elements $x, y\in\mathcal{A}$. 

A linear map $d:\mathcal{A}\rightarrow\mathcal{A}$ is a {\it{derivation}} if $d\left(  xy\right)  =d\left(  x\right)  y+xd\left(  y\right)  $ holds for all $x,y\in\mathcal{A}$.  If $d$ is a derivation for Jordan product, meaning $d\left(  x\circ y\right)  =d\left(  x\right) \circ y+x\circ d\left(  y\right)  $ for all $x,y\in\mathcal{A}$, then it is called a {\it Jordan derivation}. Let us denote the set of all derivations of $\mathcal A$ by $\operatorname*{Der}\left(
\mathcal{A}\right)  $ and by  $\operatorname*{JDer}\left(  \mathcal{A} \right)  $ the set of all Jordan derivations of algebra $\mathcal{A}$. Obviously  $\operatorname*{Der}\left( \mathcal{A}\right) \subseteq \operatorname*{JDer}\left(  \mathcal{A} \right) $.  The study of conditions that force a Jordan derivation to be a derivation was initiated in the second half of the 20th century with Jacobson and Rickart and their study of Jordan homomorphisms \cite{JR}, and with Herstein \cite{Her}. I. N. Herstein \cite{Her} proved in 1957 that every Jordan derivation of a prime ring of characteristic not $2$ is a derivation and later this result was extended to semiprime rings and algebras in various directions (see e.g. \cite{Br0, B3, Cus, Sin} and references therein).
%%In 1957 Herstein \cite{Her} proved that every Jordan derivation from a prime ring of characteristic not $2$ into itself is a derivation. This result has been
%extended to different rings and algebras in various directions (see e.g. \cite{Br0, B3, Cus, Sin} and references therein); one might very roughly summarize these results by saying that proper Jordan derivations (i.e. those
%that are not derivations) from rings (algebras) into themselves are rather rare and very special.

Further, let $\operatorname*{Cent}(\mathcal{A})$ denote the set of all \emph{centralizers} of algebra $\mathcal{A}$, i.e. linear maps $f$ satisfying $f\left(  xy\right)  =f\left(  x\right) y=xf\left(  y\right)  $ for all $x,y\in\mathcal{A}$, and let $\operatorname*{JCent}(\mathcal{A})$ be the set of all \emph{Jordan centralizers}, i.e. linear maps satisfying $f\left(  x\circ y\right) =f\left(  x\right)  \circ y$ for all $x,y\in\mathcal{A}$. Clearly, every centralizer is a Jordan centralizer and, analogous to the above mentioned study of derivations, the algebras satisfying $\operatorname*{JCent}(\mathcal{A})=\operatorname*{Cent}(\mathcal{A})$ are of special interest. For example, Zalar in \cite{Z} proved that on semiprime algebras Jordan centralizers and centralizers coincide. \smallskip

In this paper we are interested in a natural generalization of Jordan centralizers and Jordan derivations defined in terms of Jordan product. We call a linear map $f:\mathcal{A}\rightarrow \mathcal{A}$ a \emph{generalized Jordan derivation} if there exist linear maps $g,h:\mathcal{A}\rightarrow\mathcal{A}$ such that
\begin{equation}
f\left(  x\right)  \circ y+x\circ g\left(  y\right)  =h\left(  x\circ
y\right)  \quad\text{for all }x,y\in\mathcal{A}.\label{e_1}%
\end{equation}
The set of all generalized Jordan derivations of algebra $\mathcal{A}$ will be denoted by $\operatorname*{GJDer}\left(  \mathcal{A}\right)  $. Main motivation for the terminology used in the present paper comes from \cite{LL}, where Leger and Luks present a systematic study of generalized derivations of Lie algebras. A generalized derivation of a Lie algebra  $\mathcal{L=}\left(  \mathcal{L},+,\left[  ,\right]  ,\cdot\right)  $ is a linear map  $f:\mathcal{L} \rightarrow\mathcal{L}$ such that there exist linear maps $g,h:\mathcal{A}\rightarrow\mathcal{A}$ satisfying 
$\left[  f\left(  x\right) ,y\right]  +\left[  x,g\left(  y\right)  \right]  =h\left(  \left[ x,y\right]  \right)  $
for all $x,y\in\mathcal{L}$.  Note that in \cite{YB} maps satisfying the special case of (\ref{e_1}), namely $h=f$, on triangular algebras were considered. In \cite{Br}  Bre\v sar defines a linear map  $h:\mathcal{A}\rightarrow\mathcal{A}$ satisfying (\ref{e_1}) to be a Jordan  $\left\{  f,g\right\}  $-derivation. It follows from \cite[Theorem $4.3$]{Br} that in case of a unital semiprime algebra $\mathcal A$ a map $h$ is of the form $h\left(  x\right)  =\lambda x+d\left(  x\right)  $ for all $x\in\mathcal{A}$, where $d$ is a derivation and $\lambda$ is a central element. Therefore we can write $h\in \operatorname*{Cent}(\mathcal{A})+\operatorname*{Der} \left(  \mathcal{A}\right)  $.

Some connections regarding generalized Jordan derivations and Jordan $\{f,g\}$-derivations are presented in the second section. Let us note here that there exist generalized Jordan derivations on the algebra of all upper triangular matrices over a unital commutative ring which can not be presented as Jordan $\{f,g\}$-derivations. Theorem \ref{Prop_PraQ} states, that in case of a unital semiprime algebra $\mathcal A$ equality $\operatorname*{GJDer}\left( \mathcal{A}\right)  =\operatorname*{Cent}(\mathcal{A})+\operatorname*{Der}\left(  \mathcal{A}\right)  $ holds. More generally, Theorem
\ref{Theor1} states that every $f\in\operatorname*{GJDer}\left(  \mathcal{A}%
\right)  $ can be presented as $f=f_{1}+f_{2}$, where $f_{1}$ is a \emph{quasi Jordan centralizer} (i.e. map $f_{1}$ is such that $f_{1}\left( x\right)  \circ y=x\circ f_{1}\left(  y\right)  $ for all $x,y\in\mathcal{A}$)
and $f_{2}$ is a \emph{quasi Jordan derivation} (i.e. map $f_{2}$ satisfyies $f_{2}\left(  x\right)  \circ y+x\circ f_{2}\left(  y\right) =h\left(  x\circ y\right)  $ for all $x,y\in\mathcal{A}$). Denote the set of all quasi Jordan centralizers by $\operatorname*{QJCent}(\mathcal{A})$  and the set of all quasi Jordan derivations of $\mathcal A$ by $\operatorname*{QJDer}\left(  \mathcal{A}\right)  $. Note that $\operatorname*{QJCent}(\mathcal{A})\cap\operatorname*{QJDer}\left( \mathcal{A}\right)  =\operatorname*{JCent}(\mathcal{A})$ and that the following chains of inclusions hold%
\begin{gather*}
\operatorname*{Cent}\left(  \mathcal{A}\right)  \subseteq\operatorname*{JCent}%
\left(  \mathcal{A}\right)  \subseteq\operatorname*{QJCent}\left(
\mathcal{A}\right)  ,\\
\operatorname*{Der}\left(  \mathcal{A}\right)  \subseteq\operatorname*{JDer}%
\left(  \mathcal{A}\right)  \subseteq\operatorname*{JCent}\left(
\mathcal{A}\right)  +\operatorname*{JDer}\left(  \mathcal{A}\right)
\subseteq\operatorname*{QJDer}\left(  \mathcal{A}\right).
\end{gather*}
 In what follows a Jordan centralizer that is not a centralizer will be called a {\it proper } Jordan centralizer. Analogously, a proper quasi Jordan centralizer is a quasi Jordan centralizer not contained in  $\operatorname*{JCent}({\mathcal A})$.

In the third section we present a classification of maps in $\operatorname*{QJCent}\left(  \mathcal{A}\right)  $, see Theorem \ref{Theor2} and Corollary \ref{Cor1}. Matrix algebras and semiprime algebras are basic examples of algebras satisfying $\operatorname*{QJCent}\left(  \mathcal{A}\right)  =\operatorname*{Cent} \left(  \mathcal{A}\right)  $. So, there do not exist proper Jordan centralizers and not even proper quasi Jordan centralizers of these algebras. Algebras generated by idempotents and triangular algebras are examples of algebras not having proper Jordan centralizers (there can exist proper quasy Jordan centralizers of such algebras). Note, when $\operatorname*{JCent}\left( \mathcal{A}\right)  =\operatorname*{QJCent}\left(  \mathcal{A}\right)  $ holds for an algebra $\mathcal A$, then  $\operatorname*{GJDer}\left(  \mathcal{A}\right) =\operatorname*{QJDer}\left(  \mathcal{A}\right)  $ and so for the description of generalized Jordan derivations it is  enough to know the form of quasi Jordan derivations. \smallskip

In the fourth section the equality $\operatorname*{QJDer}\left( \mathcal{A}\right)  =\operatorname*{JCent}\left(  \mathcal{A}\right) +\operatorname*{JDer}\left(  \mathcal{A}\right)  $ is observed. This equality holds for matrix algebras  (Corollary \ref{Cor2}), algebras generated by idempotents (Proposition \ref{Prop_idemQ}), triangular algebras (Corollary \ref{Prop_TriQ}) and semiprime algebras. Lemma \ref{Lem} gives an easy criterion for  $f\in\operatorname*{QJDer}\left( \mathcal{A}\right)  $ to satisfy $f\in\operatorname*{JCent}\left(  \mathcal{A} \right)  +\operatorname*{JDer}\left(  \mathcal{A}\right)  $. \smallskip

% Ni pa nam uspelo najti odgovora na zastavljeno vpra\v{s}anje v splo\v{s}nem. 

%\v{S}e ve\v{c}, na trikotnih algebrah, matri\v{c}nih algebrah in
%polpraalgebrah velja $\operatorname*{QJDer}\left(  \mathcal{A}\right)
%=\operatorname*{Cent}\left(  \mathcal{A}\right)  +\operatorname*{Der}\left(
%\mathcal{A}\right)  $.\smallskip

Let us conclude this introductory part with pointing out that our definition of a generalized Jordan derivation clearly differs from the definition of a generalized Jordan derivation presented in some other studies, for example  \cite{HQ,JL,V}. In these studies a generalized Jordan derivation is a map $f:{\mathcal A}\rightarrow {\mathcal A}$ satisfying  $f(x^2)=f(x)x+xd(x)$ for all $x \in {\mathcal A}$, where $d:{\mathcal A}\rightarrow {\mathcal A}$ is a Jordan derivation. 
%Also other generalizations of Jordan derivations have been studied and there exist different definitions of a notion 'generalized Jordan derivation' that are not related to our approach (see for example).

\section{Preliminary results}

In the following theorem a basic fact about generalized Jordan derivations is given. Namely, to give the form of generalized Jordan derivations it is enough to know the forms of quasi Jordan derivations and quasi Jordan centralizers. 

\begin{theorem}
\label{Theor1} Let $\mathcal{A}$ be a unital algebra over a field $F$ and
$\operatorname*{char}\left(  F\right)  \neq2$. Then the following equalities hold

$\left( a\right)  $ $\operatorname*{GJDer}\left(  \mathcal{A}\right)
=\operatorname*{QJCent}(\mathcal{A})+\operatorname*{QJDer}\left(
\mathcal{A}\right)  $ and

$\left( b\right)  $ $\operatorname*{JCent}(\mathcal{A}%
)=\operatorname*{QJCent}(\mathcal{A})\cap\operatorname*{QJDer}\left(
\mathcal{A}\right)  $.
\end{theorem}

\begin{proof}
Let $f\in\operatorname*{GDer}\left(  \mathcal{A}\right)  $. Therefore  (\ref{e_1}) holds for some linear maps  $g,h$ on $\mathcal A$. Since Jordan product is commutative it follows that 
\begin{align*}
f\left(  x\right)  \circ y+x\circ g\left(  y\right)   &  =h\left(  x\circ
y\right)  ,\\
g\left(  x\right)  \circ y+x\circ f\left(  y\right)   &  =h\left(  x\circ
y\right)
\end{align*}
for all $x,y\in\mathcal{A}$. Denote $f_{1}=\frac{1}{2}\left(  f+g\right)  $
and $f_{2}=\frac{1}{2}\left(  f-g\right)  $. From the above identities it follows that 
%(multiplying by $\frac{1}{2}$ and adding up / subtracting)
\begin{align*}
f_{1}\left(  x\right)  \circ y+x\circ f_{1}\left(  y\right)   &  =h\left(
x\circ y\right)  ,\\
f_{2}\left(  x\right)  \circ y-x\circ f_{2}\left(  y\right)   &  =0
\end{align*}
for all $x,y\in\mathcal{A}$. Therefore $f_{1}\in\operatorname*{QJDer}\left(
\mathcal{A}\right)  $ and $f_{2}\in\operatorname*{QJCent}(\mathcal{A})$. Since
$f=f_{1}+f_{2}$ this proves $\left(  a\right)  $.

To prove $\left(  b\right)  $  let $f\in\operatorname*{QJCent}(\mathcal{A})\cap\operatorname*{QJDer}%
\left(  \mathcal{A}\right)  $. Then there exists a linear map $h$ on $\mathcal{A}$ such that
\begin{align*}
f\left(  x\right)  \circ y+x\circ f\left(  y\right)   &  =h\left(  x\circ
y\right)  \quad\text{and}\\
f\left(  x\right)  \circ y-x\circ f\left(  y\right)   &  =0
\end{align*}
for all $x,y\in\mathcal{A}$. Adding up these identites we obtain
$2f\left(  x\right)  \circ y=h\left(  x\circ y\right)  $ and further, substituting $y=1$,
we get $h\left(  x\right)  =2f\left(  x\right)  $. Therefore $f\left(  x\circ
y\right)  =f\left(  x\right)  \circ y$ for all $x,y\in\mathcal{A}$ and thus
$f\in\operatorname*{JCent}(\mathcal{A})$.
\end{proof}

Let linear maps $f,g,h:\mathcal{A}\rightarrow\mathcal{A}$ be such that (\ref{e_1}) holds. Using the definition given in \cite{Br} this is equal to saying that $h$ is a Jordan $\left\{  f,g\right\}  $-derivation.
%Tako so posplo\v{s}ena jordanska odvajanja in Jordan $\left\{  f,g\right\}  $-derivations povezana. 
Knowing the form of generalized Jordan derivations of some algebra $\mathcal A$ the form of Jordan $\{f,g\}$-derivations can be given. Namely, for $y=1$ in (\ref{e_1}) we get
\[
h\left(  x\right)  =f\left(  x\right)  +\beta\circ x\quad\text{for all }%
x\in\mathcal{A},
\]
where $2\beta=g\left(  1\right)  $. The next simple example of a map on algebra $\mathcal{A}=T_{2}\left(
A\right)  $ illustrates that the opposite does not hold. \bigskip

\noindent\textbf{Example 1.} Let $\mathcal{A}=T_{2}\left(  A\right)  $ be an
upper triangular matrix algebra over a unital commutative algebra $A$. Using the notation $e_{ij}$ for the matrix units 
%$\left\{  e_{11},e_{12},e_{22}\right\}  $ 
every $x\in\mathcal{A}$ can be presented as $x=ae_{11}+me_{12}+be_{22}$ for some $a,m,b\in A$. Define a map $f:\mathcal{A}\rightarrow\mathcal{A}$ as $f\left(  x\right)  =e_{12}\circ x=\left(  a+b\right)  e_{12}$ for all $x\in\mathcal{A}$. Direct calculations show that
\begin{align*}
f\left(  x\right)  \circ y &  =\left(  a+b\right)  e_{12}\circ y=\left(
a+b\right)  \left(  a^{\prime}+b^{\prime}\right)  e_{12},\\
x\circ f\left(  y\right)   &  =x\circ\left(  a^{\prime}+b^{\prime}\right)
e_{12}=\left(  a+b\right)  \left(  a^{\prime}+b^{\prime}\right)  e_{12},\\
f\left(  x\circ y\right)   &  =2\left(  aa^{\prime}+bb^{\prime}\right)  e_{12}%
\end{align*}
for all $x=ae_{11}+me_{12}+be_{22}$ and $y=a^{\prime}e_{11}+m^{\prime}%
e_{12}+b^{\prime}e_{22}.$ Therefore $f$ is a quasi Jordan centralizer and is not a Jordan centralizer.

Next we show that $f$ in not a Jordan $\{g,h\}$-derivation for any choice of linear maps $g,h$. Assume contrary, that $f\left(  x\circ y\right)  =g\left( x\right)  \circ y+x\circ h\left(  y\right)  $ for all $x,y\in\mathcal{A}$.
Writing $f_{1}=\frac{1}{2}\left(  g+h\right)  $ this equality transforms into $f\left(  x\circ
y\right)  =f_{1}\left(  x\right)  \circ y+x\circ f_{1}\left(  y\right)  $ for all $x,y\in\mathcal{A}$. Substituting $x=y=1$ we get $4f_{1}\left( 1\right)  =2f\left(  1\right)  =4e_{12}$ and $f_{1}\left(  1\right)  =e_{12}$.
Substitution $x=e_{11}$ and $y=e_{22}$ gives
\begin{align*}
f\left(  e_{11}\circ e_{22}\right)   &  =f_{1}\left(  e_{11}\right)  \circ
e_{22}+e_{11}\circ f_{1}\left(  e_{22}\right) \\
0  &  =f_{1}\left(  e_{11}\right)  e_{22}+e_{22}f_{1}\left(  e_{11}\right)
+e_{11}f_{1}\left(  e_{22}\right)  +f_{1}\left(  e_{22}\right)  e_{11}.
\end{align*}
Multiplying this equation from the left side by $e_{11}$ and from the right side by $e_{22}$ we get a contradiction
\[
0=e_{11}f_{1}\left(  e_{11}\right)  e_{22}+e_{11}f_{1}\left(  e_{22}\right)
e_{22}=e_{11}f_{1}\left(  1\right)  e_{22}=e_{12}
\]
and the claim is thereby proved. \bigskip

More generally, observing the above example it is not hard to see that for $n\geq 2$ the map  $f:T_{n}\left(  A\right)  \rightarrow T_{n}\left(  A\right)  $ defined as $f\left(  x\right)  =e_{1n}\circ x$ for all $x\in T_{n}\left(
A\right)  $ is a quasi Jordan centralizer and is not a Jordan centralizer. Therefore the following remark can be stated.
\begin{remark}\label{Rem}
Let $\mathcal{A}=T_{n}\left(  A\right)  $, $n\geq 2$, be an upper triangular
matrix algebra over a unital commutative algebra $A$. Then
$\operatorname*{JCent}\left(  \mathcal{A}\right)  \neq\operatorname*{QJCent}%
\left(  \mathcal{A}\right)  $.
\end{remark}

The algebra of upper triangular matrices is not a semiprime algebra. Recall, an algebra $\mathcal A$ is a semiprime algebra if $a\mathcal{A}a=\left\{  0\right\}  $, where $a\in\mathcal{A}$, implies $a=0$. If algebra $\mathcal A$ is such that $a\mathcal{A}b=\left\{  0\right\}  $, where $a,b\in\mathcal{A}$, implies $a=0$ or $b=0$ it is called a prime algebra. In case of a semiprime algebra $\mathcal{A}$ Bre\v sar's result \cite[Theorem $4.3$]{Br} can be applied and the next statement can be obtained.

\begin{theorem}
\label{Prop_PraQ} Let $\mathcal{A}$ be a unital semiprime algebra over a field $F$ with
$\operatorname*{char}\left(  F\right)  \neq2$. Then $\operatorname*{GJDer}%
\left(  \mathcal{A}\right)  =\operatorname*{Cent}\left(  \mathcal{A}\right)
+\operatorname*{Der}\left(  \mathcal{A}\right)  $.
\end{theorem}

\begin{proof}
Let $f\in\operatorname*{GJDer}\left(  \mathcal{A}\right)  $. Then there exist linear maps $g,h:\mathcal{A}\rightarrow\mathcal{A}$ such that
$f\left(  x\right)  \circ y+x\circ g\left(  y\right)  =h\left(  x\circ
y\right)  $ for all $x,y\in\mathcal{A}$. The map $h$ is a Jordan $\left\{
f,g\right\}$-derivation and by \cite[Theorem $4.3$]{Br} we know that $h$ is a $\left\{
f,g\right\}  $-derivation. Therefore
\begin{equation}
f\left(  x\right)  y+xg\left(  y\right)  =h\left(  xy\right)  =g\left(
x\right)  y+xf\left(  y\right)  \label{e_5}%
\end{equation}
for all $x,y\in\mathcal{A}$ and $h\left(  x\right)  =h\left(  1\right)
x+d\left(  x\right)  $, where $h\left(  1\right)  \in Z\left(
\mathcal{A}\right)  $ and $d:\mathcal{A}\rightarrow\mathcal{A}$ is a derivation.
Substituting $x=y=1$ in (\ref{e_5}) gives $h\left(  1\right)  =f\left(
1\right)  +g\left(  1\right)  $. Next, again using (\ref{e_5}), we have
\[
f\left(  x\right)  +xg\left(  1\right)  =h\left(  x\right)  =g\left(
1\right)  x+f\left(  x\right),
\]
thus $g\left(  1\right)  \in Z\left(  \mathcal{A}\right)  $ and therefore also $f\left(  1\right)  \in Z\left(
\mathcal{A}\right)  $. Finally, we derive
\[
f\left(  x\right)  =h\left(  x\right)  -g\left(  1\right)  x=\left(  h\left(
1\right)  -g\left(  1\right)  \right)  x+d\left(  x\right)  =f\left(
1\right)  x+d\left(  x\right)
\]
for all $x\in\mathcal{A}$ and so $f\in\operatorname*{Cent}\left(  \mathcal{A}\right)
+\operatorname*{Der}\left(  \mathcal{A}\right)  $.
\end{proof}

\section{Characterization of maps in $\operatorname*{QJCent}\left( \mathcal{A}\right)  $}

As pointed out already in the introductory section, for every unital algebra $\mathcal{A}$ over a field $F$ the inclusions $\operatorname*{Cent}\left(  \mathcal{A}\right)
\subseteq\operatorname*{JCent}\left(  \mathcal{A}\right)  \subseteq
\operatorname*{QJCent}\left(  \mathcal{A}\right)  $ hold. In this section we give a characterization of maps from these sets and provide examples of algebras for which the above inclusions are actually equalities, i.e. $\operatorname*{JCent}\left(  \mathcal{A}\right)  =\operatorname*{Cent}\left( \mathcal{A}\right)  $, $\operatorname*{QJCent}\left(  \mathcal{A}\right) =\operatorname*{JCent}\left(  \mathcal{A}\right)  $ and $\operatorname*{QJCent} \left(  \mathcal{A}\right)  =\operatorname*{Cent}\left(  \mathcal{A}\right) $. Let $\left[  x,y\right]  =xy-yx$ denote the Lie product of elements $x,y\in\mathcal{A}$.\smallskip

Recall, a map $f$ is a centralizer of a unital algebra $\mathcal A$ if and only if it is of the form $f\left(  x\right)  =ax$ where $a\in Z\left(  {\mathcal A}\right)  $.  The center of algebra $\mathcal A$, $Z\left(  \mathcal{A}\right)  =\left\{  a\in\mathcal{A}|\left[ a,\mathcal{A}\right]  =\left\{  0\right\}  \right\}  $, is therefore in a bijective corespondence with the set  $\operatorname*{Cent}\left(  \mathcal{A}\right)  $. In the characterization of the maps in $\operatorname*{QJCent}\left(  \mathcal{A} \right)  $ similar role will be played by the sets
\begin{align*}
Z_{J}\left(  \mathcal{A}\right)   &  =\left\{  a\in\mathcal{A}|\left[  \left[
a,x\right]  ,y\right]  =0\text{ for all }x,y\in\mathcal{A}\right\}
\;\text{and}\\
Z_{Q}\left(  \mathcal{A}\right)   &  =\left\{  a\in\mathcal{A}|\left[
a,\left[  x,y\right]  \right]  =0\text{ for all }x,y\in\mathcal{A}\right\}  .
\end{align*}
By the Jacobi identity, $\left[  \left[  a,x\right]
,y\right]  +\left[ \left[  y,a\right],x  \right]  +\left[ \left[
x,y\right], a  \right]  =0$ for all $a,x,y\in\mathcal{A}$, the following holds
\[
F\cdot1\subseteq Z\left(  \mathcal{A}\right)  \subseteq Z_{J}\left(
\mathcal{A}\right)  \subseteq Z_{Q}\left(  \mathcal{A}\right)  \subseteq
\mathcal{A}.
\]

\begin{theorem}
\label{Theor2} Let $\mathcal{A}$ be a unital algebra over a field $F$,
$\operatorname*{char}\left(  F\right)  \neq2$. For a linear map $f:\mathcal{A} \rightarrow\mathcal{A}$ let us denote $f\left(  1\right)  =2\alpha$. Then 

$\left(  a\right)  $ $f\in\operatorname*{QJCent}\left(  \mathcal{A}\right)  $
if and only if $f(x)=\alpha \circ x$, where $\alpha\in Z_{Q}\left(  \mathcal{A}\right)  ,$

$\left(  b\right)  $ $f\in\operatorname*{JCent}\left(  \mathcal{A}\right)  $
if and only if  $f(x)=\alpha \circ x$, where $\alpha\in Z_{J}\left(  \mathcal{A}\right)  .$
\end{theorem}

\begin{proof}
Let $f\in\operatorname*{QJCent}\left(  \mathcal{A}\right)  $. For $y=1$ in  $f\left(  x\right)  \circ y=x\circ f\left(  y\right)  $ we get $f\left(  x\right)  =\alpha\circ x$ for all $x\in\mathcal{A}$. Therefore
\[
f\left(  x\right)  \circ y-x\circ f\left(  y\right)  =\left(  \alpha\circ
x\right)  \circ y-x\circ\left(  \alpha\circ y\right)  =0
\]
for all $x,y\in\mathcal{A}$. By a known identity $\left( \alpha\circ x\right)  \circ y-x\circ\left(  \alpha\circ y\right)  =\left[ \alpha,\left[  x,y\right]  \right]  $ this implies $\alpha\in Z_{Q}\left( \mathcal{A}\right)  $. For the opposite direction of $(a)$ note that for every $\alpha\in Z_{Q}\left( \mathcal{A}\right)  $ a linear map $f$ defined as $f\left(  x\right)  =\alpha\circ x$ for all $x\in\mathcal{A}$ is contained in  $\operatorname*{QJCent}\left(  \mathcal{A}\right) $.

Next, let $f\in\operatorname*{JCent}\left(  \mathcal{A}\right)  $. Then $f\left(  x\circ y\right)  =f\left(  x\right)  \circ y$ and so
\begin{align*}
f\left(  x\circ y\right)  -f\left(  x\right)  \circ y  &  =\alpha\circ\left(
x\circ y\right)  -\left(  \alpha\circ x\right)  \circ y\\
0  &  =\left[  \left[  \alpha,y\right]  ,x\right]
\end{align*}
holds for all $x,y\in\mathcal{A}$. Therefore $\alpha\in Z_{J}\left(  \mathcal{A} \right)  $. Contrary, for $\alpha\in Z_{J}\left( \mathcal{A}\right)  $ a linear map $f\left(  x\right)  =\alpha\circ x$, $x\in\mathcal{A}$, is a Jordan centralizer.
\end{proof}

A direct corollary of the theorem can be stated.

\begin{corollary}
\label{Cor1} Let $\mathcal{A}$ be a unital algebra over a field $F$,
$\operatorname*{char}\left(  F\right)  \neq2$. Then the following propositions hold

$\left( a \right)  $ If $Z_{J}\left(  \mathcal{A}\right)  =Z\left(
\mathcal{A}\right)  $, then $\operatorname*{JCent}\left(  \mathcal{A}%
\right)  =\operatorname*{Cent}\left(  \mathcal{A}\right)  $.

$\left( b\right)  $ If $Z_{Q}\left(  \mathcal{A}\right)  =Z\left(
\mathcal{A}\right)  $, then $\operatorname*{QJCent}\left(  \mathcal{A}%
\right)  =\operatorname*{Cent}\left(  \mathcal{A}\right)  $.

$\left( c\right)  $ If $Z_{Q}\left(  \mathcal{A}\right)
=Z_{J}\left(  \mathcal{A}\right)  $, then $\operatorname*{QJCent}\left(
\mathcal{A}\right)  =\operatorname*{JCent}\left(  \mathcal{A}\right)  $.
\end{corollary}

\begin{remark}
Last corollary and Theorem \ref{Theor1} implie: if $Z_{Q}\left(  \mathcal{A}\right)  =Z_{J}\left(  \mathcal{A}\right)
$, then $\operatorname*{GJDer}\left(  \mathcal{A}\right) =\operatorname*{QJDer}\left(  \mathcal{A}\right)  $.
\end{remark}
Note, that all algebras in the next subsections are algebras over a field $F$ with $\operatorname*{char} (F)\neq2$ even if this is not mentioned explicitly. 

\subsection{Examples of algebras satisfying $\operatorname*{QJCent}\left( \mathcal{A}\right)  =\operatorname*{Cent}\left(  \mathcal{A}\right)  $}$\ $

Matrix algebras and semiprime algebras are basic classes of algebras on which there are no proper Jordan centralizers nor proper quasi Jordan centralizers.

\begin{proposition}
Let $A$ be a  unital algebra,  and $\mathcal{A}=M_{n}\left(  A\right)  $, $n\geq2$, the matrix algebra. Then $\operatorname*{QJCent}\left(  \mathcal{A} \right)  =\operatorname*{Cent}\left(  \mathcal{A}\right)  $.  
\end{proposition}

\begin{proof}
Let $\left\{  e_{ij}|i,j=1,2,...,n\right\}  $ be the set of matrix units and let $1$ denote the identity matrix in $\mathcal{A}$. Let a matrix $\mathfrak{a}%
=\sum_{i,j=1}^{n}a_{ij}e_{ij}$ satisfy $\left[
\mathfrak{a},\left[  x,y\right]  \right]  =0$ for all $x,y\in\mathcal{A}$. In particular, for every $i\neq j$ we have
\begin{align*}
0 &  =\left[  \mathfrak{a},\left[  e_{ii},e_{ij}\right]  \right]
=\mathfrak{a}e_{ij}-e_{ij}\mathfrak{a}\\
&  \mathfrak{=}\sum_{k=1}^{n}a_{ki}e_{kj}-\sum_{k=1}^{n}a_{jk}e_{ik}.
\end{align*}
Therefore $a_{ij}=0$ and $a_{ii}=a_{jj}$ for all $i\neq j$. It follows that  $\mathfrak{a}=a\cdot1$ for some $a\in A$ and it remains to prove that
$a\in Z\left(  A\right)  $. This is true since for every $b\in A$ we have
\[
0=\left[  \mathfrak{a},\left[  be_{11},e_{12}\right]  \right]  =\left[
\mathfrak{a,}be_{12}\right]  =\left(  ab-ba\right)  e_{12}.
\]
So $\mathfrak{a}=a\cdot1\in Z\left(  A\right)  \cdot1=Z\left(  \mathcal{A}\right)  $ and we have proved that $Z_{Q}\left(
\mathcal{A}\right)  =Z\left(  \mathcal{A}\right)  $. The conclusion of the proposition follows using Corollary \ref{Cor1}.
\end{proof}

Further, it can be shown that an element $a$ from a semiprime algebra $\mathcal{A}$ satisfying $\left[  a,\left[  \mathcal{A} ,\mathcal{A}\right]  \right]  =\left\{  0\right\}  $  is contained in $ Z\left( \mathcal{A}\right)  $. Therefore $Z_{Q}\left( \mathcal{A}\right)  =Z\left(  \mathcal{A}\right)  $ holds in this case and the following proposition holds (it is also an easy consequence of Theorem \ref{Prop_PraQ}).

\begin{proposition}
\label{Prop_Pra} Let $\mathcal A$ be a semiprime unital algebra. Then $\operatorname*{QJCent}\left(  \mathcal{A}\right)  =\operatorname*{Cent} \left(  \mathcal{A}\right)  $.
\end{proposition}

Let us point out two important examples of operator algebras that are semiprime. Denote by $\mathcal{B}(X)$ the algebra of all bounded linear operators on $X$, where  $X$ is a Banach space over the field of real or complex numbers. Recall, $\mathcal{A}$ is a standard operator algebra if $\mathcal{A}$ is a unital subalgebra of $\mathcal{B}(X)$ containing all finite rank operators. Every standard operator algebra is prime and as such also semiprime. 
Next, a von Neumann algebra $\mathcal{A}$ is a unital $C^{\ast} $-subalgebra of $\mathcal{B}(H)$, where $H$ is a complex Hilbert space, which is closed in the strong operator topology. Any von Neumann algebra is semiprime. 

%These operator algebras can therefore not  Tako omenjene operatorske algebre ne premorejo niti pravih
%jordanskih centralizatorjev niti pravih kvazijordanskih centralitzorjev. 
%Note that on semiprime algebras there do not exist proper Jordan centralizers nor proper quasi Jordan centralizers.

\subsection{Examples of algebras satisfying  $\operatorname*{JCent}\left(
\mathcal{A}\right)  =\operatorname*{Cent}\left(  \mathcal{A}\right)  $}$\ $

As basic examples of algebras on which there are no proper Jordan centralizers we present algebras generated by idempotents and triangular algebras. Note, since these algebras are not necessarily semiprime, there can exist proper quasi Jordan centralizers of them (see Remark \ref{Rem})). 
%Some particular classes of triangular algebras are nest algebras $\mathcal{T} \left(  \mathcal{N}\right)  $, block upper triangular algebras $B_{n}\left( A\right)  $ and upper triangular matrix algebras $T_{n}(A)$, where $A$ is a unital algebra. 
Let $R\left(  \mathcal{A}\right)  $ denote the subalgebra of algebra $\mathcal A$ generated by all idempotents in $\mathcal A$.  

\begin{proposition}
\label{Prop_idem} Let $\mathcal{A}=R\left(  \mathcal{A}\right)  $. Then
$\operatorname*{JCent}\left(  \mathcal{A}\right)  =\operatorname*{Cent}\left( \mathcal{A}\right)  $.
\end{proposition}

\begin{proof}
In view of Corollary \ref{Cor1} it suffices to prove that $Z_{J}\left( \mathcal{A}\right) =Z\left(  \mathcal{A}\right)  $. Let $e=e^{2}\in \mathcal{A}$ be a nontrivial idempotent   and denote $e^{\perp}=1-e$. Note that $ee^{\perp}=0=e^{\perp}e$. Let $a\in\mathcal{A}$ be such that $\left[  \left[  a,x\right]  ,y\right]  =0$ for all $x,y\in\mathcal{A}$. In particular
\[
0=[\left[  a,e\right]  , e^{\perp}]=[ae-ea,e^{\perp}]=-eae^{\perp}-e^{\perp
}ae.
\]
Multiplying this identity by $e$ and $e^{\perp}$ successively, we get $eae^{\perp}=0=e^{\perp}ae$.  Therefore $a=(e+e^{\perp})a(e+e^{\perp})=eae+e^{\perp}ae^{\perp}$ and further 
\[
ea=eae=ae,
\]
showing that $a$ commutes with every idempotent in $\mathcal A$. Since algebra $\mathcal A$ is generated by idempotents it follows that $ax=xa$ for all $x\in\mathcal{A}$. So $Z_{J}\left( \mathcal{A}\right)  =Z\left(  \mathcal{A}\right)  $.
\end{proof}

Next, let us turn our attention to some special algebras containing nontrivial idempotents. Assume that $\mathcal{A}$ has an idempotent $e\neq0,1$. Then $\mathcal{A}$ can be represented in the so called Peirce decomposition
\begin{equation}
\mathcal{A}=e\mathcal{A}e+e\mathcal{A}e^{\perp}+e^{\perp}\mathcal{A}%
e+e^{\perp}\mathcal{A}e^{\perp}, \label{Lie0_1}%
\end{equation}
where $e\mathcal{A}e$ and $e^{\perp}\mathcal{A}e^{\perp}$ are subalgebras with units $e$ and $e^{\perp}$, respectively, $e\mathcal{A}e^{\perp}$ is an $(e\mathcal{A}e,e^{\perp}\mathcal{A}e^{\perp})$-bimodule and $e^{\perp }\mathcal{A}e$ is an $(e^{\perp}\mathcal{A}e^{\perp},e\mathcal{A}e)$-bimodule. 
Let us assume that $\mathcal{A}$ satisfies the following conditions
\begin{equation}
\begin{aligned}\label{Lie0_2}
exe\cdot e\mathcal{A}e^{\perp}  &  =\{0\}=e^{\perp}\mathcal{A}e\cdot
exe\quad\text{implies\quad}exe=0,\\
e\mathcal{A}e^{\perp}\cdot e^{\perp}xe^{\perp}  &  =\{0\}=e^{\perp}xe^{\perp
}\cdot e^{\perp}\mathcal{A}e\quad\text{implies\quad}e^{\perp}xe^{\perp
}=0
\end{aligned}
\end{equation}
for all $x\in\mathcal{A}$.  Basic examples of unital algebras with nontrivial idempotents having the property (\ref{Lie0_2}) are triangular algebras (satisfying $e^{\perp}\mathcal{A}e=\{0\}$), matrix algebras, and prime
(hence in particular simple) algebras with nontrivial idempotents.

\begin{proposition}
\label{Prop_e_idem}Let $\mathcal{A}$ be a unital algebra with a nontrivial idempotent $e$ satisfying (\ref{Lie0_2}). Then $\operatorname*{JCent} \left(  \mathcal{A}\right)  =\operatorname*{Cent}\left( \mathcal{A}\right)  $.
\end{proposition}

\begin{proof}
Let $e\in {\mathcal A}$ be a nontrivial idempotent and decompose $\mathcal{A}$ as (\ref{Lie0_1}). Let an element $\mathfrak{a}\in\mathcal{A}$ satisfy $\left[  \left[  \mathfrak{a} ,x\right]  ,y\right]  =0$ for all $x,y\in\mathcal{A}$. Then $\left[ \mathfrak{a,}\left[  x,y\right]  \right]  =0$ holds for all $x,y\in\mathcal{A}$ and, using $\left[  e,e\mathcal{A}e^{\perp}\right]  =e\mathcal{A}e^{\perp}$ and $\left[ e^{\perp}\mathcal{A}e,e\right]  =e^{\perp}\mathcal{A}e$, it follows that $\left[ \mathfrak{a,}e\mathcal{A}e^{\perp}\right]  =\left\{  0\right\}  $ and $\left[ \mathfrak{a,}e^{\perp}\mathcal{A}e\right]  =\left\{  0\right\}  $. Thus, for arbitrary $a=eae\in e\mathcal{A}e$, $m=eme^{\perp}\in e\mathcal{A}e^{\perp}$
and $n=e^{\perp}ne\in e^{\perp}\mathcal{A}e$ we have
\begin{align*}
0 &  =\left[  \mathfrak{a,}am\right]  =\left[  \mathfrak{a,}a\right]
m+a\left[  \mathfrak{a,}m\right]  =\left[  \mathfrak{a,}a\right]  m,\\
0 &  =\left[  \mathfrak{a,}na\right]  =\left[  \mathfrak{a,}n\right]
a+n\left[  \mathfrak{a,}a\right]  =n\left[  \mathfrak{a,}a\right]  .
\end{align*}
Therefore $\left[  \mathfrak{a,}a\right]  \cdot e\mathcal{A}e^{\perp }=\{0\}=e^{\perp}\mathcal{A}e\cdot\left[  \mathfrak{a,}a\right]  $ holds for all $a=eae\in e\mathcal{A}e$ and  by (\ref{Lie0_2}) it follows that $\left[
\mathfrak{a,}e\mathcal{A}e\right]  =\left\{  0\right\}  $. Similarly we can prove that $\left[  \mathfrak{a,}e^{\perp}\mathcal{A}e^{\perp }\right]  =\left\{  0\right\}  $.
Using the obtained relations we derive
\[
\lbrack\mathfrak{a,}\mathcal{A}]=[\mathfrak{a,}e\mathcal{A}e+e\mathcal{A}%
e^{\perp}+e^{\perp}\mathcal{A}e+e^{\perp}\mathcal{A}e^{\perp}]=\left\{
0\right\},
\]
proving  $Z_{J}\left( \mathcal{A}\right)  =Z\left(  \mathcal{A}\right)  $. 

%implying $\mathfrak{a}\in Z\left(  \mathcal{A}\right)  $. 
%Since (as seen in the proof of Proposition \ref{Prop_idem}) $\mathfrak{a}$ commutes with $e$ we have  $\mathfrak{a}=e\mathfrak{a}e+e^{\perp}\mathfrak{a}e^{\perp}$.

\end{proof}

According to the last proposition and the paragraph preceding it, the next corollary can be stated.
\begin{corollary}
\label{Prop_Tri}Let $\mathcal{A}$ be a triangular algebra. Then
$\operatorname*{JCent}\left(  \mathcal{A}\right)  =\operatorname*{Cent}\left(
\mathcal{A}\right)  $.
\end{corollary}

Since every upper triangular matrix algebra $\mathcal{A}=T_{n}\left(  A\right)  $, $n\geq2$, where $A$ is a unital algebra, can be represented as a triangular algebra,
\[
\mathcal{A}=\left(
\begin{array}
[c]{cc}%
A & M_{1\times\left(  n-1\right)  }\left(  A\right) \\
& T_{n-1}\left(  A\right)
\end{array}
\right) ,
\]
by Corollary \ref{Prop_Tri} there do not exist proper Jordan centralizers of $T_{n}\left(  A\right)$. 
%$\operatorname*{JCent} \left(  \mathcal{A}\right)  =\operatorname*{Cent}\left(  \mathcal{A}\right)  $
%
%\begin{corollary}
%Let $\mathcal{A}=T_{n}\left(  A\right)  $, $n\geq2$, be an upper triangular
%matrix algebra over a unital algebra $A$. Potem je $\operatorname*{JCent}%
%\left(  \mathcal{A}\right)  =\operatorname*{Cent}\left(  \mathcal{A}\right)  $.
%\end{corollary}

Another class of triangular algebras are nest algebras $\mathcal{T}\left(  \mathcal{N}\right)$ associated to a nest $\mathcal{N} \neq\left\{  0,H\right\}  $ (see \cite[Proposition 5]{Ch}). Let us recall, a \emph{nest} is a chain $\mathcal{N}$ of closed subspaces of a complex Hilbert space $H$ containing $\left\{  0\right\}  $ and $H$ which is closed under arbitrary intersections and closed linear span. The \emph{nest algebra} associated to $\mathcal{N}$ is the algebra
\[
\mathcal{T}\left(  \mathcal{N}\right)  =\left\{  T\in\mathcal{B}\left(
H\right)  \mid T\left(  N\right)  \subseteq N\text{ for all }N\in
\mathcal{N}\right\}  .
\]
If $\mathcal{N}=\left\{  0,H\right\}  $, then $\mathcal{T}\left( \mathcal{N}\right)  =\mathcal{B}\left(  H\right)  $ is an algebra of all bounded linear operators on $H$ and it is a prime algebra. Let us mention that finite dimensional nest algebras are isomorphic to a complex block upper triangular matrix algebras.
By Proposition \ref{Prop_Pra} and Corollary \ref{Prop_Tri} the following corollary holds.

\begin{corollary}
Let $\mathcal{N}$ be nest of a complex Hilbert space $H$, $\dim H\geq2$ and
let $A=\mathcal{T}\left(  \mathcal{N}\right)  $ be a nest algebra. Then
$\operatorname*{JCent}\left(  \mathcal{A}\right)  =\operatorname*{Cent}\left(
\mathcal{A}\right)  .$
\end{corollary}

\subsection{Examples of algebras satisfying  $\operatorname*{QJCent}\left( \mathcal{A}\right)  =\operatorname*{JCent}\left(  \mathcal{A}\right)  $}$\ $

An algebra $\mathcal{A}$ satisfying $\operatorname*{QJCent} \left(  \mathcal{A}\right)  =\operatorname*{JCent}\left(  \mathcal{A}\right) \neq\operatorname*{Cent}\left(  \mathcal{A}\right)  $ can be constructed as follows. Let $\mathcal{A}$ be a noncommutative algebra that satisfies polynomial identity $\left[  \left[ X,Y\right]  ,Z\right]  $. Then for every $a\in\mathcal{A}$ we have $\left[  \left[  a,\mathcal{A}\right]  ,\mathcal{A}\right]  =\left\{
0\right\}  $ and $\left[  a,\left[  \mathcal{A},\mathcal{A}\right]
\right]  =\left\{  0\right\}  $. Therefore $Z\left(  \mathcal{A}\right)
\neq\mathcal{A}=Z_{J}\left(  \mathcal{A}\right)  =Z_{Q}\left(  \mathcal{A}%
\right)  $ holds and every map $f:\mathcal{A}\rightarrow\mathcal{A}$ defined as $f\left(  x\right)  =a\circ x$, where $a\notin Z\left( \mathcal{A}\right)  $, is a proper Jordan centralizer.

An example of algebra satisfying the conditions from previous paragraph is a Grassmann algebra. Namely, a Grassmann algebra $\mathcal{A}$ is a $\mathbb{Z}_{2}$-graded algebra of the form $\mathcal{A}=\mathcal{A}_{0}+\mathcal{A}_{1}$ where $\mathcal{A} _{0}=Z\left(  \mathcal{A}\right)  $ and $\left[  \mathcal{A},\mathcal{A} \right]  =\left[  \mathcal{A}_{1},\mathcal{A}_{1}\right]  \subseteq \mathcal{A}_{0}=Z\left(  \mathcal{A}\right)  $. An example of such algebra is%
\[
\mathcal{A}=\left\{
\begin{bmatrix}
r & s & u\\
& r & t\\
&  & r
\end{bmatrix}
;r,s,t,u\in A\right\}  \subseteq T_{3}\left(  A\right)  ,
\]
where $A$ is a commutative unital algebra. Note that 
\[
\mathcal{A}_{0}=Z\left(  \mathcal{A}\right)  =\left\{
\begin{bmatrix}
r & 0 & u\\
& r & 0\\
&  & r
\end{bmatrix}
;r,u\in A\right\}  ,\;\mathcal{A}_{1}=\left\{
\begin{bmatrix}
0 & s & 0\\
& 0 & t\\
&  & 0
\end{bmatrix}
;s,t\in A\right\}  .
\]

\subsection{An algebra satisfying  $\operatorname*{Cent}\left(  \mathcal{A}\right)
\neq\operatorname*{JCent}\left(  \mathcal{A}\right)  \neq
\operatorname*{QJCent}\left(  \mathcal{A}\right)  \label{Primer0}$} $\ $

As the final example of this section we present an algebra $\mathcal{A}$ satisfying the chain of strong inclusions $Z\left(  \mathcal{A}\right)  \subset Z_{J}\left(  \mathcal{A}\right)  \subset
Z_{Q}\left(  \mathcal{A}\right)  .$ Note that this algebra will be used in the construction of a counterexample presented in the next section. This algebra, given already in \cite{YB}, is subalgebra of the matrix algebra $M_{4}\left(  A\right)  $, where $A$ is a commutative unital algebra. Let 
\begin{equation}
\mathcal{A}=\left\{
\begin{bmatrix}
r & s & t_{1} & u\\
& r & v & t_{2}\\
&  & r & s\\
&  &  & r
\end{bmatrix}
;\ r,s,t_{i},u,v\in A\right\}  \subseteq M_{4}\left(  A\right)  .\label{Primer}%
\end{equation}
Direct calculations show 
\begin{align*}
Z\left(  \mathcal{A}\right)   &  =\left\{  r1+t\left(  e_{13}+e_{24}\right)
+ue_{14}|r,t,u\in A\right\}  ,\\
Z_{J}\left(  \mathcal{A}\right)   &  =\left\{  r1+t_{1}e_{13}+t_{2}%
e_{24}+ue_{14}|r,t_{i},u\in A\right\}  ,\\
Z_{Q}\left(  \mathcal{A}\right)   &  =\left\{  r1+t_{1}e_{13}+t_{2}%
e_{24}+ue_{14}+ve_{23}|r,t_{i},u,v\in A\right\}  ,
\end{align*}
noting that $e_{13}$ is a noncentral element in $Z_{J}\left(  \mathcal{A}\right)  $ and that  $e_{23}\in Z_{Q}\left( \mathcal{A}\right) \setminus Z_{J}\left(  \mathcal{A}\right)  $. It follows that the map $f_{1}\left(  x\right)  =e_{13}\circ x$ is a proper Jordan centralizer and the map $f_{2}\left(  x\right)  =e_{23}\circ x$ is a proper quasi Jordan centralizer.

\section{Characterization of maps in $\operatorname*{QJDer}\left( \mathcal{A}\right)  $}

Recalling the definitions of (quasi) Jordan derivations and Jordan centralizers it is not hard to see that for every unital algebra $\mathcal A$ the inclusion $\operatorname*{JCent}\left(  \mathcal{A}\right)  +\operatorname*{JDer}\left( \mathcal{A}\right)  \subseteq\operatorname*{QJDer}\left(  \mathcal{A}\right)$ holds. Here we will be interested in unital algebras over a field $F$, $\operatorname*{char} (F)\neq2$,  satisfying $\operatorname*{QJDer}\left(  \mathcal{A}\right)  =\operatorname*{JCent} \left(  \mathcal{A}\right)  +\operatorname*{JDer}\left(  \mathcal{A}\right) $. Using the observations from the previous section some examples of algebras satisfying $\operatorname*{QJDer}\left(  \mathcal{A}\right)  =\operatorname*{Cent} \left(  \mathcal{A}\right)  +\operatorname*{Der}\left(  \mathcal{A}\right) $ will also be given. 

%Kot smo videli v prej\v{s}njem podpoglavju za veliko primerov algeber velja
%$\operatorname*{JCent}\left(  \mathcal{A}\right)  =\operatorname*{Cent}\left(
%\mathcal{A}\right)  $, prav tako so jordanska odvajanja v veliko primerih
%hkrati odvajana. 

First we note that $\operatorname*{JCent}\left(  \mathcal{A}\right)  \cap \operatorname*{JDer}\left(  \mathcal{A}\right)  =\left\{  0\right\}  $. This holds true since every $f\in\operatorname*{JCent}(\mathcal{A})\cap \operatorname*{JDer}\left(  \mathcal{A}\right)  $ satisfies
\[
f\left(  x\right)  \circ y+x\circ f\left(  y\right)  =f\left(  x\circ
y\right)  =f\left(  x\right)  \circ y
\]
for all $x,y\in\mathcal{A}$. Therefore $x\circ f\left(  y\right)  =0$ and, for $x=1$,  $f=0$ follows.

Now let $f\in\operatorname*{QJDer}\left(  \mathcal{A}\right)  $ and let us denote $f\left(  1\right)  =2\alpha$. Set $d\left(  x\right) =f\left(  x\right)  -\alpha\circ x$ for all $x\in\mathcal{A}$. Then
$f\left(  x\right)  =\alpha\circ x+d\left(  x\right)  $ and $f\in
\operatorname*{JCent}\left(  \mathcal{A}\right)  +\operatorname*{JDer}\left(
\mathcal{A}\right)  $ if and only if the map $\bar{f}\left( x\right)  =\alpha\circ x$ is a Jordan centralizer and $d$ is a Jordan derivation. Of course we can not expect that $\operatorname*{QJDer}\left(  \mathcal{A}\right)  =\operatorname*{JCent} \left(  \mathcal{A}\right)  +\operatorname*{JDer}\left(  \mathcal{A}\right)  $
holds in general. Let us give an example of algebra for which
$\operatorname*{QJDer}\left(  \mathcal{A}\right)  \neq\operatorname*{JCent}%
\left(  \mathcal{A}\right)  +\operatorname*{JDer}\left(  \mathcal{A}\right)
$.\bigskip

\noindent\textbf{Example.} Let $\mathcal{A}$ be the algebra defined in (\ref{Primer}). Let the maps $f,h:\mathcal{A}\rightarrow\mathcal{A}$ be given as
\[%
\begin{bmatrix}
r & s & t_{1} & u\\
& r & v & t_{2}\\
&  & r & s\\
&  &  & r
\end{bmatrix}
\overset{f}{\mapsto}%
\begin{bmatrix}
0 & 0 & 0 & 0\\
& 0 & 2r & 0\\
&  & 0 & 0\\
&  &  & 0
\end{bmatrix}
\;\text{and}\;%
\begin{bmatrix}
r & s & t_{1} & u\\
& r & v & t_{2}\\
&  & r & s\\
&  &  & r
\end{bmatrix}
\overset{h}{\mapsto}%
\begin{bmatrix}
0 & 0 & s & 0\\
& 0 & 4r & s\\
&  & 0 & 0\\
&  &  & 0
\end{bmatrix}
.
\]
Calculations show that $f\left(  x\right)  \circ y+x\circ f\left(  y\right)  =h\left(  x\circ y\right)  $ for all $x,y\in\mathcal{A}$. Therefore $f$ is a quasi Jordan derivation. Since $f\left( 1\right)  =2e_{23}$ we can  write $f\left(  x\right) =e_{23}\circ x+d\left(  x\right)  $. Thus, the map $\bar{f}\left(  x\right)
=e_{23}\circ x$ is not a Jordan centralizer as seen in subsection \ref{Primer0}. Also, the map $d\left(  x\right)  =f\left( x\right)  -e_{23}\circ x=-s\left(  e_{13}+e_{24}\right)  $ is not a Jordan derivation. Indeed, since $x\circ x=0$ holds for $x=e_{12}+e_{34}$ it follows that $0= d\left(  x\circ x\right)  \neq d\left(  x\right)  \circ x+x\circ d\left(  x\right)  =-4e_{14}$, which is a contradiction. Therefore $f\notin\operatorname*{JCent}\left(  \mathcal{A}\right)
+\operatorname*{JDer}\left(  \mathcal{A}\right)  $.\bigskip

The most we can say in general is the following criterion.

\begin{lemma}
\label{Lem} Let $f\in\operatorname*{QJDer}\left(  \mathcal{A}\right)  $.
Then $f\in\operatorname*{JCent}\left(  \mathcal{A}\right) +\operatorname*{JDer}\left(  \mathcal{A}\right)  $ if and only if $f\left(  1\right)  \in Z_{J}\left(  A\right)  $.
\end{lemma}

\begin{proof}
Clearly, if $f=\bar{f}+d\in\operatorname*{JCent}\left(  \mathcal{A}\right) +\operatorname*{JDer}\left(  \mathcal{A}\right)  $, then $f\left( 1\right)  =\bar{f}\left(  1\right)  $ and, since $\bar{f}$ is a Jordan centralizer, we have $f\left(  1\right)  \in Z_{J}\left(  A\right)  $.

For the proof of the other implication let $f\in\operatorname*{QJDer}\left(  \mathcal{A}\right)  $. Then there exist $h:\mathcal{A}\rightarrow\mathcal{A}$ such that
\begin{equation}
f\left(  x\right)  \circ y+x\circ f\left(  y\right)  =h\left(  x\circ
y\right)  \label{e_3}%
\end{equation}
for all $x,y\in\mathcal{A}$. Denote $f\left(  1\right)  =2\alpha$. Substituting $y=1$ in (\ref{e_3}) we get $2f\left(  x\right)  +x\circ f\left(  1\right)  =2h\left(  x\right)  $ or equivalently $f\left(  x\right)
+\alpha\circ x=h\left(  x\right)  $ for all $x\in\mathcal{A}$. Therefore
\begin{equation}
f\left(  x\right)  \circ y+x\circ f\left(  y\right)  =f\left(  x\circ y\right)  +\alpha\circ\left(  x\circ y\right)  \label{e_4}
\end{equation}
for all $x,y\in\mathcal{A}$. Let us write $f\left(  x\right)  =\alpha\circ x+d\left(  x\right)  $, where $d:\mathcal{A}\rightarrow\mathcal{A}$. By (\ref{e_4}) we have
\[
\left(  \alpha\circ x\right)  \circ y+d\left(  x\right)  \circ y+x\circ\left(
\alpha\circ y\right)  +x\circ d\left(  y\right)  =d\left(  x\circ y\right)
+2\left(  \alpha\circ\left(  x\circ y\right)  \right)  ,
\]
which can be written as
\begin{align*}
d\left(  x\circ y\right)  -d\left(  x\right)  \circ y-x\circ d\left(
y\right)   &  =\left(  \alpha\circ x\right)  \circ y+\left(  \alpha\circ
y\right)  \circ x-2\left(  \alpha\circ\left(  x\circ y\right)  \right) \\
&  =\left[  \left[  \alpha,x\right]  ,y\right]  +\left[  \left[
\alpha,y\right]  ,x\right]
\end{align*}
for all $x,y\in\mathcal{A}$. Now, if $f\left(  1\right)  =2\alpha\in
Z_{J}\left(  \mathcal{A}\right)  $, then $\left[  \left[  \alpha
,\mathcal{A}\right]  ,\mathcal{A}\right]  =\left\{  0\right\}  $ and it follows that $\bar
{f}\left(  x\right)  =\alpha\circ x$ is a Jordan centralizer and $d$ is a Jordan derivation. Therefore  $f=\bar{f}+d\in\operatorname*{JCent}\left( \mathcal{A}\right)  +\operatorname*{JDer}\left(  \mathcal{A}\right)  $.
\end{proof}

Thus, to know whether $\operatorname*{QJDer}\left(  \mathcal{A}\right)=\operatorname*{JCent}\left( \mathcal{A}\right)  +\operatorname*{JDer}\left(  \mathcal{A}\right)  $ holds for an algebra $\mathcal A$ we have to check if every $f\in\operatorname*{QJDer}\left(  \mathcal{A}\right)  $ satisfies $f\left( 1\right)  \in Z_{J}\left( \mathcal{A}\right)  $ and identity (\ref{e_4}) can be used to verify this. In case of matrix algebras, algebras generated by idempotents, algebras containing a nontrivial idempotent $e$ satisfying (\ref{Lie0_2}) and semiprime algebras it can be shown that $f\left(  1\right)  \in Z\left(  \mathcal{A}\right)  $. Since on some algebras every Jordan derivation is a derivation, some refined statements can be made and the following results are obtained. 

\begin{proposition}
\label{Prop_idemQ} Let $\mathcal{A}=R\left(  \mathcal{A}\right)  $. Then $\operatorname*{QJDer}\left(  \mathcal{A}\right)  =\operatorname*{Cent}\left( \mathcal{A}\right)  +\operatorname*{JDer}\left(  \mathcal{A}\right)  $.
\end{proposition}

\begin{proof}
Let $f\in\operatorname*{QJDer}\left(  \mathcal{A}\right)  $, $f\left( 1\right)  =2\alpha$, and let $e$ be an idempotent of algebra $\mathcal{A}$. Set $e^{\perp}=1-e$. Substituting $y=x=e$ in (\ref{e_4}) we get $2f\left(  e\right)  e+2ef\left(  e\right)     =2f\left(  e\right)  +2\left(
\alpha\circ e\right)$ or equivalently
\begin{align*}
f\left(  e\right)  e+ef\left(  e\right)   &  =f\left(  e\right)  +\alpha e+ea.
\end{align*}
Multiplying this equality from the left with $e$ and from the right with $e^{\perp}$ we further get
\[
ef\left(  e\right)  e^{\perp}=ef\left(  e\right)  e^{\perp}+e\alpha e^{\perp
}.
\]
It follows that $e\alpha e^{\perp}=0$ and similarly it can be proven that  $e^{\perp }\alpha e=0$.  Therefore $\alpha=e\alpha e+e^{\perp}\alpha e^{\perp}$ and, noticing that $\alpha e=e\alpha e=e\alpha$ holds, we have derived that $\alpha$ commutes with idempotent $e$. By assumption $\mathcal{A}=R\left(  \mathcal{A}\right)  $, so  $\alpha\in Z\left( \mathcal{A}\right)  $. The conclusion of the proposition follows by Lemma \ref{Lem}.
\end{proof}

\begin{corollary}
\label{Cor2} Let $A$ be a unital algebra and $\mathcal{A}=M_{n}\left(  A\right) $, $n\geq2$, the matrix algebra. Then $\operatorname*{QJDer}\left( \mathcal{A}\right)  =\operatorname*{Cent}\left(  \mathcal{A}\right) +\operatorname*{Der}\left(  \mathcal{A}\right)  $.
\end{corollary}

\begin{proof}
Every matrix algebra over a unital algebra is generated by idempotents. Therefore the statement of Proposition \ref{Prop_idemQ} holds for $M_{n}\left(  A\right)$. Further, from classical results of Jacobson and Rickart \cite[Theorems 7 and 22]{JR} it follows that every Jordan derivation of $M_{n}\left(  A\right)  $ is a derivation. 
\end{proof}

\begin{proposition}
Let $\mathcal{A}$ be a unital algebra with a nontrivial idempotent $e$ satisfying (\ref{Lie0_2}). Then $\operatorname*{QJDer}\left( \mathcal{A}\right)  =\operatorname*{Cent}\left(  \mathcal{A}\right) +\operatorname*{JDer}\left(  \mathcal{A}\right)  $.
\end{proposition}

\begin{proof}
Let $f\in\operatorname*{QJDer}\left(  \mathcal{A}\right)  $, $f\left(
1\right)  =2\alpha$. Let us use the decomposition (\ref{Lie0_1}) for $\mathcal{A}$. As seen in the proof of 
Proposition \ref{Prop_idemQ} the element $\alpha$ commutes with all idempotents of algebra $\mathcal{A}$. The equality $\left[ \alpha ,e\right]  =0$ implies $\alpha=e\alpha e+e^{\perp }\alpha e^{\perp}$. Note that $e+m$ and $e+n$ are idempotents of $\mathcal{A}$ for all $m=eme^{\perp}\in e\mathcal{A}e^{\perp}$ and $n=e^{\perp }ne\in e^{\perp}\mathcal{A}e$. Therefore we have
\[
\left[  \alpha,e\right]  =0,\text{\quad}\left[  \alpha,e+m\right]
=0,\text{\quad}\left[  \alpha,e+n\right]  =0
\]
and so $\left[  \alpha,m\right]  =0=\left[  \alpha ,n\right]  $ for all $m\in e\mathcal{A}e^{\perp}$ and $n\in e^{\perp }\mathcal{A}e$, or equivalently $\left[  \alpha, e\mathcal{A}e^{\perp }\right]  =\left\{  0\right\}  =\left[  \alpha, e^{\perp}\mathcal{A} e\right]  $. By the same arguments as in the proof of Proposition \ref{Prop_e_idem} it can be seen that this implies $\left[  \alpha, e\mathcal{A} e\right]  =\left\{  0\right\}  =\left[  \alpha, e^{\perp}\mathcal{A} e^{\perp}\right]  $ and $\alpha \in Z\left(  \mathcal{A}\right)  $. It follows by Lemma \ref{Lem} that $f\in\operatorname*{Cent}\left(  \mathcal{A}\right) +\operatorname*{JDer}\left(  \mathcal{A}\right)  $.
\end{proof}

Since every Jordan derivation of a triangular algebra is a derivation, this was proven by Zhang and Yu in \cite{ZY}, we can state the next corollary.

\begin{corollary}
\label{Prop_TriQ} Let a be triangular algebra $\mathcal{A}$. Then $\operatorname*{QJDer}\left(  \mathcal{A}\right)  =\operatorname*{Cent}\left( \mathcal{A}\right)  +\operatorname*{Der}\left(  \mathcal{A}\right)  $.
\end{corollary}

Every upper triangular matrix algebra $T_{n}\left(  A\right)  ,$ $n\geq2$, over a unital algebra $A$ is a triangular algebra.  The nest algebras $\mathcal{T}\left(  \mathcal{N}\right)  \neq\mathcal{B}\left(  H\right)  $,  $\dim H\geq2$, are triangular algebras. Therefore Corollary \ref{Prop_TriQ} holds for these algebras. It can be seen that the same conclusion holds also for algebra $\mathcal{B} \left(  H\right)  $ since it is prime. Namely, as a direct consequence of the Theorem \ref{Prop_PraQ} the following proposition can be stated.

\begin{corollary}
Let $A$ be a unital semiprime algebra. Then $\operatorname*{QJDer}\left( \mathcal{A}\right) =\operatorname*{Cent}\left(  \mathcal{A}\right) +\operatorname*{Der}\left(  \mathcal{A}\right)  $.
\end{corollary}

\bigskip

\end{document}